\documentclass[12pt]{article}
\usepackage[psamsfonts]{amssymb}
\usepackage{amsthm,amsmath}

\title{Representations and characterizations of polynomial functions on chains}

\author{Miguel Couceiro \\
Mathematics Research Unit, University of Luxembourg \\
162A, avenue de la Fa\"{\i}encerie, L-1511 Luxembourg, Luxembourg \\
miguel.couceiro[at]uni.lu
\and %
Jean-Luc Marichal \\
Mathematics Research Unit, University of Luxembourg \\
162A, avenue de la Fa\"{\i}encerie, L-1511 Luxembourg, Luxembourg \\
jean-luc.marichal[at]uni.lu }

\date{Revised, May 25, 2009}

\begin{document}
\maketitle

\theoremstyle{plain}
\newtheorem{theorem}{Theorem}
\newtheorem{lemma}[theorem]{Lemma}
\newtheorem{proposition}[theorem]{Proposition}
\newtheorem{corollary}[theorem]{Corollary}
\newtheorem{fact}[theorem]{Fact}

\theoremstyle{definition}
\newtheorem{definition}[theorem]{Definition}
\newtheorem{example}[theorem]{Example}

\theoremstyle{remark}
\newtheorem{conjecture}{Conjecture}
\newtheorem{remark}{Remark}

\newcommand{\card}[1]{\ensuremath{\lvert{#1}\rvert}}
\newcommand{\vect}[1]{\ensuremath{\mathbf{#1}}} 
\newcommand{\co}[1]{\ensuremath{\overline{#1}}}
\def\median{\mathop{\rm med}\nolimits}

\begin{abstract}
We are interested in representations and characterizations of lattice polynomial functions $f\colon L^n\to L$, where $L$ is a given bounded
distributive lattice. In companion papers \cite{CouMar,CouMar1}, we investigated certain representations and provided various characterizations of
these functions both as solutions of certain functional equations and in terms of necessary and sufficient conditions. In the present paper, we
investigate these representations and characterizations in the special case when $L$ is a chain, i.e., a totally ordered lattice. More
precisely, we discuss representations of lattice polynomial functions given in terms of standard simplices and we present new axiomatizations of
these functions by relaxing some of the conditions given in \cite{CouMar,CouMar1} and by considering further conditions, namely comonotonic minitivity
and maxitivity.
\end{abstract}


\section{Introduction}

In \cite{CouMar,CouMar1}, the class of (\emph{lattice}) \emph{polynomial functions}, i.e., functions representable by combinations of variables and
constants using the lattice operations $\wedge$ and $\vee$, was considered  and characterized
both as solutions of certain functional equations  and in terms of necessary and sufficient conditions rooted in aggregation theory.

Formally, let $L$ be a bounded distributive lattice with operations $\wedge$ and $\vee$, and with least and greatest elements $0$ and $1$,
respectively. An \emph{$n$-ary polynomial function} on $L$ is any function $f\colon L^n\to L$ which can be obtained by finitely many applications of
the following rules:
\begin{enumerate}
\item[(i)] For each $i\in [n]=\{1,\ldots,n\}$ and each $c\in L$, the projection $\vect{x}\mapsto x_i$ and the constant function $\vect{x}\mapsto
c$ are polynomial functions from $L^n$ to $L$.

\item[(ii)] If $f$ and $g$ are polynomial functions from $L^n$ to $L$, then $f\vee g$ and $f\wedge g$ are polynomial functions from $L^n$ to
$L$.
\end{enumerate}
Polynomial functions are also called lattice functions (Goodstein~\cite{Goo67}), algebraic functions (Burris and Sankappanavar~\cite{BurSan81})
or weighted lattice polynomial functions (Marichal~\cite{Marc}). Polynomial functions
 obtained from projections by finitely many applications of (ii) are referred to as (\emph{lattice}) \emph{term functions}.  As an example, we have the ternary \emph{median function} \begin{eqnarray*}
\median(x,y,z) &=& (x\vee y)\wedge (y\vee z)\wedge (z\vee x)\\
&=& (x\wedge y)\vee (y\wedge z)\vee (z\wedge x).
\end{eqnarray*}

The recent interest by aggregation theorists in this class of polynomial functions is partially motivated by its connection to noteworthy
aggregation functions such as the (\emph{discrete}) \emph{Sugeno integral}, which was introduced by Sugeno~\cite{Sug74,Sug77} and widely
investigated in aggregation theory, due to the many applications in fuzzy set theory, data fusion, decision making, image analysis, etc. As
shown in \cite{Marc}, the discrete Sugeno integrals are nothing other than those polynomial functions $f\colon L^n\to L$ which are
idempotent, that is, satisfying $f(x,\ldots,x)=x$. For general background on aggregation theory, see \cite{BelPraCal07,GraMurSug00} and for a
recent reference, see \cite{GraMarMesPap09}.

In this paper, we refine our previous results in the particular case when $L$ is a chain, by relaxing our conditions and proposing weak
analogues of those properties used in \cite{CouMar,CouMar1}, and then providing characterizations of polynomial functions, accordingly. Moreover, and
motivated by the axiomatizations of the discrete Sugeno integrals established by de Campos and Bola\~{n}os~\cite{deCBol92} (in the case when
$L=[0,1]$ is the unit real interval), we present further and alternative characterizations of polynomial functions given in terms of comonotonic
minitivity and maxitivity. As particular cases, we consider the subclass of discrete Sugeno integrals as well as that of term functions.

The current paper is organized as follows. We start in {\S}2 by introducing the basic notions needed in this paper and presenting the
characterizations of lattice polynomial functions on arbitrary (possibly infinite) bounded distributive lattices, established in \cite{CouMar,CouMar1}.
Those characterizations are reassembled in Theorem~\ref{mainChar}. We discuss representations of polynomial functions in normal form (such as
the classical disjunctive and conjunctive normal forms) and introduce variant representations in the case of chains and given in terms of
standard simplices in {\S}3. In {\S}4, we provide characterizations of polynomial functions on chains given in terms of weak analogues of the
properties used in Theorem~\ref{mainChar} as well as in terms of comonotonic minitivity and maxitivity. The subclasses of discrete Sugeno
integrals and of term functions are then axiomatized in {\S}5 using the results obtained in the previous section.


\section{Basic notions and terminology}

Throughout this paper, let $L$ be a bounded distributive lattice with operations $\wedge$ and $\vee$, and with least and greatest elements $0$
and $1$, respectively. For $a,b\in L$, $a\leqslant b$ simply means that $a\wedge b=a$ or, equivalently, $a\vee b=b$.
 A \emph{chain} is simply a lattice such that for every $a,b\in L$ we have $a\leqslant b$ or
$b\leqslant a$. A subset $S$ of a lattice $L$ is said to be \emph{convex} if for every $a,b\in S$ and every $c\in L$ such that $a\leqslant
c\leqslant b$, we have $c\in S$. For any subset $S\subseteq L$, we denote by $\co{S}$ the convex hull of $S$, that is, the smallest convex
subset of $L$ containing $S$. For every $a,b\in S$ such that $a\leqslant b$, we denote by $[a,b]$ the \emph{interval} $[a,b]=\{c\in L: a\leqslant c\leqslant b\}$. For any integer $n\geqslant 1$, we set $[n]=\{1,\ldots,n\}$.

For an arbitrary nonempty set $A$ and a lattice $L$, the set $L^A$ of all functions from $A$ to $L$ constitutes a lattice under the operations
$$
(f\wedge g)(x)=f(x)\wedge g(x) \quad \textrm{ and } \quad (f\vee g)(x)=f(x)\vee g(x),
$$
for every $f,g\in L^A$. In particular, any lattice $L$ induces a lattice structure on the Cartesian product $L^n$, $n\geqslant 1$, by defining
$\wedge$ and $\vee$ componentwise, i.e.,
\begin{eqnarray*}
(a_1,\ldots ,a_n)\wedge (b_1,\ldots ,b_n) &=& (a_1\wedge b_1, \ldots , a_n\wedge b_n), \\
(a_1,\ldots ,a_n)\vee (b_1,\ldots ,b_n) &=& (a_1\vee b_1, \ldots , a_n\vee b_n).
\end{eqnarray*}
We denote the elements of $L$ by lower case letters $a,b,c,\ldots$, and the elements of $L^n$, $n>1$, by bold face letters
$\vect{a},\vect{b},\vect{c},\ldots$. We also use $\vect{0}$ and $\vect{1}$ to denote the least element and greatest element, respectively, of
$L^n$. For $c\in L$ and $\vect{x}=(x_1,\ldots ,x_n)\in L^n$, set
$$
\vect{x}\wedge c=(x_1\wedge c, \ldots , x_n\wedge c) \quad \textrm{and} \quad \vect{x}\vee c=(x_1\vee c, \ldots , x_n\vee c).
$$

The \emph{range} of a function $f\colon L^{n}\rightarrow L$ is defined by $\mathcal{R}_f=\{f(\vect{x}) : \vect{x}\in L^n\}$. A function $f\colon
L^{n}\rightarrow L$ is said to be \emph{nondecreasing} (\emph{in each variable}) if, for every $\vect{a}, \vect{b}\in L^n$ such that
$\vect{a}\leqslant\vect{b}$, we have $f(\vect{a})\leqslant f(\vect{b})$. Note that if $f$ is nondecreasing, then
$\co{\mathcal{R}}_f=[f(\vect{0}),f(\vect{1})]$.

Let $S$ be a nonempty subset of $L$. A function $f\colon L^{n}\rightarrow L$ is said to be
\begin{itemize}
\item \emph{$S$-idempotent} if for every $c\in S$, we have $f(c,\ldots,c)=c$.

\item \emph{$S$-min homogeneous} if for every $\vect{x}\in L^n$ and every $c\in S$, we have
\begin{equation}\label{minHom}
f(\vect{x}\wedge c) = f(\vect{x})\wedge c.
\end{equation}

\item \emph{$S$-max homogeneous} if for every $\vect{x}\in L^n$ and every $c\in S$, we have
\begin{equation}\label{maxHom}
f(\vect{x}\vee c) = f(\vect{x})\vee c.
\end{equation}

\item \emph{horizontally $S$-minitive} if for every $\vect{x}\in L^n$ and every $c\in S$, we have
\begin{align}\label{HorMin}
f(\vect{x}) = f(\vect{x}\vee c)\wedge f([\vect{x}]^c),
\end{align}
where $[\vect{x}]^c$ is the $n$-tuple whose $i$th component is $1$, if $x_i\geqslant c$, and $x_i$, otherwise.

\item \emph{horizontally $S$-maxitive} if for every $\vect{x}\in L^n$ and every $c\in S$, we have
\begin{align}\label{HorMax}
 f(\vect{x}) = f(\vect{x}\wedge c)\vee f([\vect{x}]_c),
\end{align}
where $[\vect{x}]_c$ is the $n$-tuple whose $i$th component is $0$, if $x_i\leqslant c$, and $x_i$, otherwise.

\item \emph{median decomposable} if, for every $\vect{x}\in L^n$ and every $k\in [n]$, we have
\begin{align}\label{MedDecomposition}
f(\vect{x})=\median\big(f(\vect{x}^{0}_{k}), x_k, f(\vect{x}^{1}_{k})\big),
\end{align}
where $\vect{x}^{c}_{k} = (x_1,\ldots, x_{k-1},c,x_{k+1},\ldots ,x_n)$ for any $c\in L$.

\item \emph{strongly idempotent} if, for every $\vect{x}\in L^n$ and every $k\in [n]$, we have
$$
f(x_1,\ldots,x_{k-1},f(\vect{x}),x_{k+1},\ldots,x_n)=f(\vect{x}).
$$
\end{itemize}

\begin{remark}
In the case when $S=L$ is the real interval $[0,1]$, the concepts of $S$-min and $S$-max homogeneity were used by Fodor and
Roubens~\cite{FodRou95} to specify certain classes of aggregation functions (for an earlier reference, see Bassanezi and Greco~\cite{BasGre88}), and the concept of horizontal $S$-maxitivity was introduced by
Benvenuti et al.~\cite{BenMesViv02} as a general property of the Sugeno integral. The concept of median decomposability was introduced in \cite{Marc} and that of strong idempotency in \cite{CouMar,CouMar1} as properties of polynomial functions.
\end{remark}

We say that a function $f\colon L^{n}\rightarrow L$ has a \emph{componentwise convex range} if either $n=1$ and $f$ has a convex range, or $n>1$
and for every $\vect{a}\in L^n$ and every $k\in [n]$, the unary function $f_{\vect{a}}^k\colon L\to L$, given by $f_{\vect{a}}^k(x)=f(\vect{a}_k^x)$ has a convex range.

The following theorem reassembles the various characterizations of polynomial functions, established in \cite{CouMar}, in the particular case when $L$ is a bounded chain.

\begin{theorem} \label{mainChar}
Let $L$ be a bounded chain and let $f\colon L^{n}\rightarrow L$ be a function. The following conditions are equivalent:
\begin{enumerate}
\item[(i)] $f$ is a polynomial function.

\item[(ii)] $f$ is median decomposable.

\item[(iii)] $f$ is nondecreasing, strongly idempotent, has a convex range and a componentwise convex range.

\item[(iv)] $f$ is nondecreasing, $\co{\mathcal{R}}_f$-min homogeneous, and $\co{\mathcal{R}}_f$-max homogeneous.

\item[(v)] $f$ is nondecreasing, $\co{\mathcal{R}}_f$-min homogeneous, and horizontally $\co{\mathcal{R}}_f$-maxitive.

\item[(vi)] $f$ is nondecreasing, horizontally $\co{\mathcal{R}}_f$-minitive, and $\co{\mathcal{R}}_f$-max homogeneous.

\item[(vii)] $f$ is nondecreasing, $\co{\mathcal{R}}_f$-idempotent, horizontally $\co{\mathcal{R}}_f$-minitive, and horizontally
$\co{\mathcal{R}}_f$-maxitive.
\end{enumerate}
\end{theorem}

\section{Representations of polynomial functions}

Polynomial functions are known to be exactly those functions which can be represented by formulas in disjunctive and conjunctive normal forms.
This fact was first observed by Goodstein~\cite{Goo67} who, in fact, showed that each polynomial function $f\colon L^{n}\rightarrow L$ is
uniquely determined by its restriction to  $\{0,1\}^n$. For a recent reference, see Rudeanu~\cite{Rud01}.

In this section we recall and refine some known results concerning normal forms of polynomial functions and, in the special case when $L$ is a
chain, we provide variant representations given in terms of standard simplices of $L^n$.

The following three results are due to Goodstein~\cite{Goo67}.

\begin{corollary}\label{Ext1}
Every polynomial function is completely determined by its restriction to $\{0,1\}^n$.
\end{corollary}

\begin{corollary}\label{Ext2}
A function $g\colon \{0,1\}^n\rightarrow L$ can be extended to a polynomial function $f\colon L^{n}\rightarrow L$ if and only if it is
nondecreasing. In this case, the extension is unique.
\end{corollary}

\begin{proposition}\label{DNF}
Let $f\colon L^{n}\rightarrow L$ be a function. The following conditions are equivalent:
\begin{enumerate}
\item[(i)] $f$ is a polynomial function.

\item[(ii)] There exists $\alpha \colon 2^{[n]}\rightarrow L$ such that $f(\vect{x})=\bigvee_{I\subseteq [n]}(\alpha(I)\wedge \bigwedge_{i\in I} x_i)$.
\item[(iii)] There exists $\beta \colon 2^{[n]}\rightarrow L$ such that $f(\vect{x})=\bigwedge_{I\subseteq [n]}(\beta(I)\vee \bigvee_{i\in I} x_i)$.
\end{enumerate}
\end{proposition}

We shall refer to the expressions given in $(ii)$ and $(iii)$ of Proposition~\ref{DNF} as the \emph{disjunctive normal form} (DNF)
representation and the \emph{conjunctive normal form} (CNF) representation, respectively, of the polynomial function $f$.

\begin{remark}
By requiring $\alpha$ and $\beta$ to be nonconstant functions from $2^{[n]}$ to $\{0,1\}$ and satisfying $\alpha(\varnothing)=0$ and
$\beta(\varnothing)=1$, respectively, we obtain the analogue of Proposition~\ref{DNF} for term functions.
\end{remark}

For each polynomial function $f\colon L^{n}\rightarrow L$, set
\begin{eqnarray*}
\mathrm{DNF}(f) &=& \Big\{\alpha \in L^{2^{[n]}} : f(\vect{x})=\bigvee_{I\subseteq [n]}\big(\alpha(I)\wedge \bigwedge_{i\in I} x_i\big)\Big\},\\
\mathrm{CNF}(f) &=& \Big\{\beta \in L^{2^{[n]}} : f(\vect{x})=\bigwedge_{I\subseteq [n]}\big(\beta(I)\vee \bigvee_{i\in I} x_i\big)\Big\}.
\end{eqnarray*}

A complete description of the sets $\mathrm{DNF}(f)$ and $\mathrm{CNF}(f)$ can be found in \cite{CouMar1}. As we
are concerned by the case when $L$ is a chain, we recall the description only in this special case; see \cite{Marc}.

For each $I\subseteq [n]$, let $\vect{e}_I$ be the element of $L^n$ whose $i$th component is $1$, if $i\in I$, and $0$, otherwise. Let $\alpha_f
\colon 2^{[n]}\rightarrow L$ be the function given by $\alpha_f(I)=f(\vect{e}_I)$ and consider the function $\alpha^{*}_f \colon
2^{[n]}\rightarrow L$ defined by
\[
\alpha^{*}_f(I) =
\begin{cases}
\alpha_f(I), & \text{if $\bigvee_{J\varsubsetneq I}\alpha_f(J) < \alpha_f(I)$,} \\
0, & \text{otherwise.}
\end{cases}
\]
Dually, Let $\beta_f \colon 2^{[n]}\rightarrow L$ be the function given by $\beta_f(I)=f(\vect{e}_{[n]\setminus I})$ and consider the function
$\beta^{*}_f \colon 2^{[n]}\rightarrow L$ defined by
\[
\beta^{*}_f(I) =
\begin{cases}
\beta_f(I), & \text{if $\bigwedge_{J \varsubsetneq
I}\beta_f(J)>\beta_f(I)$,} \\
1, & \text{otherwise.}
\end{cases}
\]

\begin{proposition}\label{prop:Uniqueness}
Let $L$ be a bounded chain and let $f\colon L^{n}\rightarrow L$ be a polynomial function.  Then
\begin{enumerate}
\item[(i)] $\mathrm{DNF}(f)=[\alpha^{*}_f,\alpha_f]$ and $f$ has a unique DNF representation if and only if $\bigvee_{J\varsubsetneq I}\alpha_f
(J)< \alpha_f(I)$ for every $I\subseteq [n]$,

\item[(ii)] $\mathrm{CNF}(f)=[\beta_f,\beta^{*}_f]$ and $f$ has a unique CNF representation if and only if $\bigwedge_{J \varsubsetneq
I}\beta_f(J)>\beta_f(I)$ for every $I\subseteq [n]$.
\end{enumerate}
In particular, $\alpha_f$ and $\beta_f$ are the unique isotone and antitone, respectively, maps in $\mathrm{DNF}(f)$ and $\mathrm{CNF}(f)$,
respectively.
\end{proposition}

In the case of chains, the DNF and CNF representations of polynomial functions $f\colon L^{n}\rightarrow L$ can be refined and given in terms of
standard simplices of $L^n$ (see Proposition~\ref{SimplexDNF} below). To provide these variants, we first need the following lemma due to Dubois
and Prade~\cite{DubPra86}. For the sake of self-containment, we provide a simpler proof. Recall that
$$
\median(x_1,\ldots ,x_{2n+1})=\bigvee_{\textstyle{I\subseteq [2n+1]\atop \card{I}=n+1}}\, \bigwedge_{i\in I}x_i.
$$

\begin{lemma}\label{SimplexMedian}
Let $a_1,\ldots ,a_{n+1},b_1, \ldots, b_{n+1} \in L$ such that $a_1\leqslant \cdots \leqslant a_{n+1}$ and $b_1\geqslant \cdots \geqslant
b_{n+1}$. If $a_{n+1}\geqslant b_{n+1}$, then
\begin{eqnarray*}
\bigvee_{i\in [n+1]}(a_i\wedge b_i) = \median(a_1,\ldots ,a_n, b_1, \ldots, b_{n+1}).
\end{eqnarray*}
\end{lemma}

\begin{proof}
Clearly, we have $\bigvee_{i\in [n+1]}(a_i\wedge b_i)\leqslant \median(a_1,\ldots ,a_n, b_1, \ldots, b_{n+1})$. Hence, to complete the proof it
is enough to show that for every $I\subseteq [n]$ and every $J\subseteq [n+1]$, with $\card{I}+\card{J}=n+1$, there is $k\in [n+1]$ such that
\begin{equation}\label{medExp}
\big(\bigwedge_{i\in I}a_{i}\big )\wedge \big (\bigwedge_{j \in J}b_{j}\big )\leqslant a_k\wedge b_k.
\end{equation}
Let $i'=\bigwedge_{i\in I}i$ and $j'=\bigvee_{j\in J}j$. Since $\card{I}+\card{J}=n+1$, we have that $[i',n]\cap [1,j']\neq \varnothing$.
Clearly, any $k\in [i',n]\cap [1,j']$ fulfills (\ref{medExp}).
\end{proof}

Let $\sigma$ be a permutation on $[n]$. The \emph{standard simplex} of $L^n$ associated with $\sigma$ is the subset $L^n_\sigma\subset L^n$
defined by
$$
L^n_\sigma =\{(x_1,\ldots ,x_n)\in L^n\colon x_{\sigma (1)}\leqslant x_{\sigma (2)}\leqslant \cdots\leqslant x_{\sigma (n)}\}.
$$
For each $i\in[n]$, define $S^{\uparrow}_\sigma (i)=\{\sigma(i),\ldots ,\sigma (n)\}$ and $S^{\downarrow}_\sigma (i)=\{\sigma(1),\ldots ,\sigma
(i)\}$. As a matter of convenience, set $S^{\uparrow}_\sigma (n+1)=S^{\downarrow}_\sigma (0)=\varnothing$.

\begin{proposition}\label{SimplexDNF}
Let $L$ be a bounded chain and let $f\colon L^{n}\rightarrow L$ be a function. The following conditions are equivalent:
\begin{enumerate}
\item[(i)] $f$ is a polynomial function.

\item[(ii)] For any permutation $\sigma$ on $[n]$ and every $\vect{x}\in L^n_\sigma$, we have
\begin{eqnarray}
 f(\vect{x}) &=& \bigvee_{i\in [n+1]}\big (\alpha_f(S^{\uparrow}_\sigma (i))\wedge x_{\sigma(i)}\big )
= \bigwedge_{i\in [n+1]}\big (\alpha_f(S^{\uparrow}_\sigma (i))\vee x_{\sigma(i-1)}\big )\label{DNF1234}\\
&=& \median\big(x_1,\ldots ,x_n, \alpha_f(S^{\uparrow}_\sigma (1)),\ldots ,\alpha_f(S^{\uparrow}_\sigma (n+1))\big),\nonumber
\end{eqnarray} where $x_{\sigma(0)}=0$ and $x_{\sigma(n+1)}=1$.

\item[(iii)] For any permutation $\sigma$ on $[n]$ and every $\vect{x}\in L^n_\sigma$, we have
\begin{eqnarray*}
f(\vect{x}) &=& \bigvee_{i\in [n+1]}\big (\beta_f(S^{\downarrow}_\sigma (i-1))\wedge x_{\sigma(i)}\big ) = \bigwedge_{i\in [n+1]}\big
(\beta_f(S^{\downarrow}_\sigma (i-1))\wedge x_{\sigma(i-1)}\big )\\
&=& \median\big(x_1,\ldots ,x_n, \beta_f(S^{\downarrow}_\sigma (0)),\ldots ,\beta_f(S^{\downarrow}_\sigma (n))\big),
\end{eqnarray*} where $x_{\sigma(0)}=0$ and $x_{\sigma(n+1)}=1$.
\end{enumerate}
\end{proposition}

\begin{proof}
The equivalence $(ii)\Leftrightarrow (iii)$ follows immediately from the fact that $\beta_f(I)=\alpha_f([n]\setminus I) $, for every $I\subseteq
[n]$, and $S^{\downarrow}_\sigma (i-1)=[n]\setminus S^{\uparrow}_\sigma (i)$.

To show that $(i)\Rightarrow (ii)$, suppose that $f\colon L^{n}\rightarrow L$ is a polynomial function, that is,
\begin{align}\label{DNF123}
f(\vect{x})=\bigvee_{I\subseteq [n]}\big(\alpha_f(I)\wedge \bigwedge_{i\in I} x_i\big).
\end{align}
Clearly, for each permutation $\sigma$ on $[n]$ and each $\vect{x}\in L^n_\sigma$, we have that (\ref{DNF123}) becomes
$$
f(\vect{x}) = \alpha_f(\varnothing)\vee \bigvee_{i\in [n]}\bigvee_{\textstyle{I\subseteq S^{\uparrow}_\sigma (i)\atop \sigma(i)\in I}}\big
(\alpha_f(I)\wedge x_{\sigma(i)}\big) = \alpha_f(\varnothing)\vee \bigvee_{i\in [n]}\big (\alpha_f(S^{\uparrow}_\sigma (i))\wedge
x_{\sigma(i)}\big).
$$
This, together with Lemma~\ref{SimplexMedian}, shows that $(i)\Rightarrow (ii)$.

To show that $(ii)\Rightarrow (i)$ also holds, let $f'\colon L^{n}\rightarrow L$ be the polynomial function given by
\begin{align*}
f'(\vect{x})=\bigvee_{I\subseteq [n]}\big(\alpha_f(I)\wedge \bigwedge_{i\in I} x_i\big).
\end{align*}
By Corollary~\ref{Ext2}, $f'$ is the unique extension of $\alpha_f$ to a polynomial function. Let $\vect{x}\in L^n$ and let $\sigma$ be
permutation on $[n]$ such that $\vect{x}\in L^n_{\sigma}$. As in the proof of $(i)\Rightarrow (ii)$, we have
\begin{align}\label{DNF123456}
f'(\vect{x})=\bigvee_{I\subseteq [n]}\big(\alpha_f(I)\wedge \bigwedge_{i\in I} x_i\big)= \bigvee_{i\in [n+1]}\big (\alpha_f(S^{\uparrow}_\sigma
(i))\wedge x_{\sigma(i)}\big)=f(\vect{x}).
\end{align}
Since (\ref{DNF123456}) holds for any $\vect{x}\in L^n$, it follows that $f=f'$ and thus $f$ is a polynomial function.
\end{proof}

\begin{remark}
The equivalence between $(i)$ and $(ii)$ of Proposition~\ref{SimplexDNF} was already observed in \cite[{\S}5]{Marc}. Prior to this,
Propositions~\ref{prop:Uniqueness} and \ref{SimplexDNF} were already established in \cite{Mar00c} for idempotent polynomial functions
(discrete Sugeno integrals) in the case when $L$ is the unit real interval $[0,1]$; see also \cite[{\S}4.3]{Mar98}.
\end{remark}

\section{Characterizations of polynomial functions}

In this section, we propose weak analogues of the properties used in Theorem~\ref{mainChar} and provide characterizations of polynomial
functions on chains, accordingly. Moreover, we introduce further properties, namely comonotonic minitivity and maxitivity, which we then use to
provide further characterizations of polynomial functions.

For integers $0\leqslant p\leqslant q\leqslant n$, define
$$
L_n^{(p,q)}=\{\vect{x}\in L^n : |\{x_1,\ldots,x_n\}\cap\{0,1\}|\geqslant p~\mbox{and}~|\{x_1,\ldots,x_n\}|\leqslant q\}.
$$
For instance, $L_n^{(0,2)}$ is the set of Boolean vectors of $L^n$ that are two-sided trimmed by constant vectors, that is
$$
L_n^{(0,2)}=\bigcup_{\textstyle{\vect{e}\in\{0,1\}^n\atop c,d\in L}}\{\median(c,\vect{e},d)\}.
$$

\subsection{Weak homogeneity}

Let $S$ be a nonempty subset of $L$. We say that a function $f\colon L^{n}\rightarrow L$ is \emph{weakly $S$-min homogeneous} (resp.\
\emph{weakly $S$-max homogeneous}) if (\ref{minHom}) (resp.\ (\ref{maxHom})) holds for every $\vect{x}\in L_n^{(0,2)}$ and every $c\in S$.

For every integer $m\geqslant 1$, every $\vect{x}\in L^m$, and every $f\colon L^{n}\rightarrow L$, we define $\langle\vect{x}\rangle_f\in L^m$
as the $m$-tuple
$$
\langle\vect{x}\rangle_f=\median(f(\vect{0}),\vect{x},f(\vect{1})),
$$
where the right-hand side median is taken componentwise. As observed in \cite{CouMar1}, for every nonempty subset $S\subseteq L$, we have that
$f$ is $S$-min homogeneous and $S$-max homogeneous if and only if it satisfies
\begin{equation}\label{Min-Max-Hom-Med}
f(\median({r},\vect{x},{s}))= \median(r,f(\vect{x}),s)
\end{equation}
for every $\vect{x}\in L^n$ and every $r,s\in S$. In particular, if $f(\vect{0}),f(\vect{1})\in S$ then, for any $\vect{x}\in L^n$ such that
$f(\vect{0})\leqslant f(\vect{x})\leqslant f(\vect{1})$, we have $f(\vect{x}) = f(\langle\vect{x}\rangle_f)$.

It was also shown in \cite{CouMar1} that, for every nonempty subset $S\subseteq L$, if $f$ is $S$-min homogeneous and $S$-max homogeneous, then
it is $S$-idempotent. The following lemma shows that the weak analogue also holds.

\begin{lemma}\label{lemma:WeaklyMinMaxIdem}
Let $S$ be a nonempty subset of $L$. If $f\colon L^{n}\rightarrow L$ is weakly $S$-min homogeneous and weakly $S$-max homogeneous, then it is
$S$-idempotent. Moreover, if $f(\vect{0}),f(\vect{1})\in S$ then, for any $\vect{x}\in L_n^{(0,2)}$ such that $f(\vect{0})\leqslant
f(\vect{x})\leqslant f(\vect{1})$, we have $f(\vect{x}) = f(\langle\vect{x}\rangle_f)$.
\end{lemma}

\begin{proof}
If $f\colon L^{n}\rightarrow L$ is weakly $S$-min homogeneous and weakly $S$-max homogeneous then, for any $c\in S$, we have $f(\vect{1}\wedge
c)\wedge c=f(\vect{1}\wedge c) = f(\vect{1}\wedge c)\vee c$, and thus $f$ is $S$-idempotent. The second statement follows from formula
(\ref{Min-Max-Hom-Med}) when restricted to vectors $\vect{x}\in L_n^{(0,2)}$.
\end{proof}

As we are going to see, of particular interest is when $S=\co{\mathcal{R}}_f$, for which we have the following result (see \cite{CouMar1}).

\begin{proposition}\label{prop:Hom-Id-46}
For any function $f\colon L^{n}\rightarrow L$ the following hold:
\begin{enumerate}
\item[(i)] If $f$ is $\co{\mathcal{R}}_f$-idempotent, then $f$ has a convex range (i.e., $\co{\mathcal{R}}_f=\mathcal{R}_f$).

\item[(ii)] If $f$ is a polynomial function, then it is $\co{\mathcal{R}}_f$-min homogeneous, $\co{\mathcal{R}}_f$-max homogeneous,
$\co{\mathcal{R}}_f$-idempotent, and has a convex range.

\item[(iii)] The function $f$ is $\co{\mathcal{R}}_f$-min homogeneous (resp.\ $\co{\mathcal{R}}_f$-max homogeneous) if and only if it is
$\mathcal{R}_f$-min homogeneous (resp.\ $\mathcal{R}_f$-max homogeneous) and has a convex range. In this case,
$\mathcal{R}_f=\co{\mathcal{R}}_f=[f(\vect{0}),f(\vect{1})]$.
\end{enumerate}
\end{proposition}

Let $f\colon L^n\to L$ be a nondecreasing function so that $\co{\mathcal{R}}_f=[f(\vect{0}),f(\vect{1})]$. Now, if $f$ is weakly
$\co{\mathcal{R}}_f$-min homogeneous, then for any $c\in \co{\mathcal{R}}_f$, we have $f(\vect{1}\wedge c)=f(\vect{1})\wedge c=c$, and thus $f$
is $\co{\mathcal{R}}_f$-idempotent. Dually, if $f$ is weakly $\co{\mathcal{R}}_f$-max homogeneous, then it is also
$\co{\mathcal{R}}_f$-idempotent. Hence we have the following result.

\begin{lemma}\label{lemma:WeaklyMinMaxRangeIdem}
Let $f\colon L^{n}\rightarrow L$ be nondecreasing. If $f$ is weakly $\co{\mathcal{R}}_f$-min homogeneous or weakly $\co{\mathcal{R}}_f$-max
homogeneous, then it is $\co{\mathcal{R}}_f$-idempotent.
\end{lemma}

We now provide our first characterization of polynomial functions which shows that, in the case of chains, the conditions in $(iv)$ of
Theorem~\ref{mainChar} can be replaced with their weak analogues.

\begin{theorem}\label{theorem:WLP-WeakHom}
Let $L$ be a bounded chain. A function $f\colon L^{n}\rightarrow L$ is a polynomial function if and only if it is nondecreasing, weakly
$\co{\mathcal{R}}_f$-min homogeneous, and weakly $\co{\mathcal{R}}_f$-max homogeneous.
\end{theorem}

\begin{proof}
From Proposition~\ref{prop:Hom-Id-46} $(ii)$ it follows that each condition is necessary. To show that they are also sufficient, let
$\vect{x}\in L^n$. By nondecreasing monotonicity, Lemma~\ref{lemma:WeaklyMinMaxIdem}, and weak $\co{\mathcal{R}}_f$-min homogeneity, for every
$I\subseteq [n]$ we have
\begin{eqnarray*}
f(\vect{x}) & \geqslant & f\big(\vect{e}_I \wedge (\bigwedge_{i\in I}x_i)\big) = f\big(\big\langle\vect{e}_I \wedge (\bigwedge_{i\in
I}x_i)\big\rangle_f\big) =  f(\langle\vect{e}_I\rangle_f) \wedge \big\langle\bigwedge_{i\in I}x_i\big\rangle_f\\
&=& f(\vect{e}_I) \wedge \big\langle\bigwedge_{i\in I}x_i\big\rangle_f
\end{eqnarray*}
and thus $f(\vect{x})\geqslant \bigvee_{I\in [n]} (f(\vect{e}_I) \wedge \langle\bigwedge_{i\in I}x_i\rangle_f)$. To complete the proof, it is
enough to establish the converse inequality. Let $I^*\subseteq [n]$ be such that $f(\vect{e}_{I^*}) \wedge\langle \bigwedge_{i\in
I^*}x_i\rangle_f$ is maximum. Define
$$
J=\big\{j\in [n]\colon x_j\leqslant f(\vect{e}_{I^*}) \wedge\big\langle \bigwedge_{i\in I^*}x_i\big\rangle_f\big\}.
$$
We claim that $J\neq\varnothing$. For the sake of contradiction, suppose that $x_j> f(\vect{e}_{I^*}) \wedge\langle\bigwedge_{i\in
I^*}x_i\rangle_f$ for every $j\in [n]$. Then, by nondecreasing monotonicity, we have $f(\vect{e}_{[n]})\geqslant f(\vect{e}_{I^*})$, and since
$f(\vect{e}_{[n]})= f(\vect{1})\geqslant\langle\bigwedge_{i\in [n]}x_i\rangle_f$,
$$
f(\vect{e}_{[n]}) \wedge \big\langle\bigwedge_{i\in [n]}x_i\big\rangle_f > f(\vect{e}_{I^*}) \wedge\big\langle\bigwedge_{i\in
I^*}x_i\big\rangle_f,
$$
which contradicts the definition of $I^*$. Thus $J\neq \varnothing$.

Now, by nondecreasing monotonicity and weak $\co{\mathcal{R}}_f$-max homogeneity, we have
$$
f(\vect{x}) \leqslant f\big (\big(f(\vect{e}_{I^*}) \wedge \big\langle\bigwedge_{i\in I^*}x_i\big\rangle_f\big)\vee \vect{e}_{[n]\setminus J}
\big) = \big(f(\vect{e}_{I^*}) \wedge\big\langle\bigwedge_{i\in I^*}x_i\big\rangle_f\big)\vee f(\vect{e}_{[n]\setminus J}).
$$
We claim that $f(\vect{e}_{[n]\setminus J})\leqslant f(\vect{e}_{I^*}) \wedge \langle\bigwedge_{i\in I^*}x_i\rangle_f$. Indeed, otherwise by
definition of $J$ we would have
$$
f(\vect{e}_{[n]\setminus J})\wedge \big\langle\bigwedge_{i\in [n]\setminus J}x_i\big\rangle_f > f(\vect{e}_{I^*}) \wedge
\big\langle\bigwedge_{i\in I^*}x_i\big\rangle_f,
$$
again contradicting the definition of $I^*$. Finally,
$$
f(\vect{x}) \leqslant f(\vect{e}_{I^*}) \wedge \big\langle\bigwedge_{i\in I^*}x_i\big\rangle_f = \bigvee_{I\in [n]} \big(f(\vect{e}_I) \wedge
\big\langle \bigwedge_{i\in I}x_i\big\rangle_f\big).\qedhere
$$
\end{proof}

\begin{remark}\label{HomogNotLattice}
\begin{enumerate}
\item[(i)] Note that Theorem~\ref{theorem:WLP-WeakHom} does not generally hold in the case of bounded distributive lattices. To see this, let
$L=\{0,a,b,1\}$ where $a\wedge b=0$ and $a\vee b=1$, and consider the binary function $f\colon L^2\rightarrow L$ defined by
\[
f(x_1,x_2) =
\begin{cases}
1, & \text{if $x_1=1$ or $x_2=1$,} \\
1, & \text{if $x_1=x_2=b$,}\\
a, & \text{if ($x_1=a$ and $x_2\neq 1$) or ($x_2=a$ and $x_1\neq 1$),}\\
0, & \text{otherwise.}
\end{cases}
\]
It is easy to verify that $f$ is nondecreasing and both weakly $\co{\mathcal{R}}_f$-min homogeneous and weakly $\co{\mathcal{R}}_f$-max
homogeneous. However, it is easy to see that $f$ is not a polynomial function.

\item[(ii)] The proof technique of Theorem~\ref{theorem:WLP-WeakHom} was already used in \cite{Mar00c} to prove a similar result
for idempotent polynomial functions (discrete Sugeno integrals) in the case when $L$ is the unit real interval $[0,1]$.
\end{enumerate}
\end{remark}

\subsection{Weak horizontal minitivity and maxitivity}

Let $S$ be a nonempty subset of $L$. We say that a function $f\colon L^{n}\rightarrow L$ is \emph{weakly horizontally $S$-minitive} (resp.\
\emph{weakly horizontally $S$-maxitive}) if (\ref{HorMin}) (resp.\ (\ref{HorMax})) holds for every $\vect{x}\in L_n^{(0,2)}$ and every $c\in S$.

\begin{lemma}\label{lemma:Weak15682}
Let $S$ be a nonempty subset of a bounded chain $L$. If $f\colon L^{n}\rightarrow L$ is nondecreasing, $S$-idempotent, and weakly horizontally $S$-minitive
(resp.\ weakly horizontally $S$-maxitive) then it is weakly $S$-min homogeneous (resp.\ weakly $S$-max homogeneous).
\end{lemma}

\begin{proof}
Let $f\colon L^{n}\rightarrow L$ be nondecreasing, $S$-idempotent, and weakly horizontally $S$-minitive. Then, for any $\vect{x}\in L_n^{(0,2)}$
and any $c\in S$,
\begin{eqnarray*}
f(\vect{x})\wedge c &=& f(\vect{x})\wedge f(c,\ldots,c) ~\geqslant ~ f(\vect{x}\wedge c) ~=~ f((\vect{x}\wedge c)\vee c)\wedge f([\vect{x}\wedge c]^c) \\
&=& f(c,\ldots,c)\wedge f([\vect{x}]^c) ~ \geqslant ~ f(c,\ldots,c)\wedge f(\vect{x}) ~=~ f(\vect{x})\wedge c.
\end{eqnarray*}
Hence $f$ is weakly $S$-min homogeneous. The other statement can be proved similarly.
\end{proof}

\begin{lemma}\label{Weak-Hor-Min-Hom}
Assume $L$ is a bounded chain. Let $f\colon L^{n}\rightarrow L$ be nondecreasing and weakly $\co{\mathcal{R}}_f$-min homogeneous (resp.\ weakly $\co{\mathcal{R}}_f$-max
homogeneous). Then $f$ is weakly $\co{\mathcal{R}}_f$-max homogeneous (resp.\ weakly $\co{\mathcal{R}}_f$-min homogeneous) if and only if it is
weakly horizontally $\co{\mathcal{R}}_f$-maxitive (resp.\ weakly horizontally $\co{\mathcal{R}}_f$-minitive).
\end{lemma}

\begin{proof}
Let $f\colon L^{n}\rightarrow L$ be nondecreasing and weakly $\co{\mathcal{R}}_f$-min homogeneous. By Lemma~\ref{lemma:WeaklyMinMaxRangeIdem},
$f$ is $\co{\mathcal{R}}_f$-idempotent.

Assume first that $f$ is also weakly $\co{\mathcal{R}}_f$-max homogeneous. For any $\vect{x}\in L_n^{(0,2)}$ and any $c\in \co{\mathcal{R}}_f$,
we have $[\vect{x}]_c\in L_n^{(0,2)}$ and hence
\begin{eqnarray*}
  f(\vect{x}\wedge c)\vee f([\vect{x}]_c)
  &=& \big(f(\vect{x})\wedge c\big)\vee f([\vect{x}]_c) ~=~ \big(f(\vect{x})\vee f([\vect{x}]_c)\big)\wedge\big(c\vee f([\vect{x}]_c)\big) \\
  &=& f(\vect{x})\wedge f(c\vee [\vect{x}]_c)%
  ~=~ f(\vect{x}).
\end{eqnarray*}
Therefore, $f$ is weakly horizontally $\co{\mathcal{R}}_f$-maxitive.

Now assume that $f$ is weakly horizontally $\co{\mathcal{R}}_f$-maxitive and let us prove that it is weakly $\co{\mathcal{R}}_f$-max
homogeneous. For any $\vect{x}\in L_n^{(0,2)}$ and any $c\in \co{\mathcal{R}}_f$, we have
$$
f(\vect{x}\vee c) = f\big((\vect{x}\vee c)\wedge c\big)\vee f([\vect{x}\vee c]_c) = f(c,\ldots,c)\vee f([\vect{x}]_c) = c \vee f([\vect{x}]_c).
$$
Therefore, we have
$$
f(\vect{x}\vee c) = f(\vect{x}\wedge c)\vee f(\vect{x}\vee c) = f(\vect{x}\wedge c) \vee f([\vect{x}]_c) \vee c = f(\vect{x})\vee c
$$
and hence $f$ is weakly $\co{\mathcal{R}}_f$-max homogeneous. The other claim can be verified dually.
\end{proof}

The following result reassembles Theorem~\ref{theorem:WLP-WeakHom}, Lemmas~\ref{lemma:Weak15682} and \ref{Weak-Hor-Min-Hom}, and provides
characterizations of the $n$-ary polynomial functions on a chain $L$, given in terms of weak homogeneity and weak horizontal minitivity and
maxitivity.

\begin{theorem}\label{theorem:WLP-WeakHomWeakHor}
Let $L$ be a bounded chain and let $f\colon L^{n}\rightarrow L$ be a function. The following conditions are equivalent:
\begin{enumerate}
\item[(i)] $f$ is a polynomial function.

\item[(ii)] $f$ is nondecreasing, weakly $\co{\mathcal{R}}_f$-min homogeneous, and weakly $\co{\mathcal{R}}_f$-max homogeneous.

\item[(iii)] $f$ is nondecreasing, weakly $\co{\mathcal{R}}_f$-min homogeneous, and weakly horizontally $\co{\mathcal{R}}_f$-maxitive.

\item[(iv)] $f$ is nondecreasing, weakly horizontally $\co{\mathcal{R}}_f$-minitive, and weakly $\co{\mathcal{R}}_f$-max homogeneous.

\item[(v)] $f$ is nondecreasing, $\co{\mathcal{R}}_f$-idempotent, weakly horizontally $\co{\mathcal{R}}_f$-minitive, and weakly horizontally
$\co{\mathcal{R}}_f$-maxitive.
\end{enumerate}
\end{theorem}

By Lemma~\ref{lemma:WeaklyMinMaxIdem} and Proposition~\ref{prop:Hom-Id-46} (ii), any lattice polynomial function $f\colon L^n\to L$ satisfies
$f(\vect{x})=f(\langle\vect{x}\rangle_f)$. Using this fact, we can adjust the proof of Lemma~\ref{Weak-Hor-Min-Hom} to replace weak horizontal
$\co{\mathcal{R}}_f$-maxitivity (resp.\ weak horizontal $\co{\mathcal{R}}_f$-minitivity) with weak horizontal $L$-maxitivity (resp.\ weak
horizontal $L$-minitivity) in Lemma~\ref{Weak-Hor-Min-Hom} and Theorem~\ref{theorem:WLP-WeakHomWeakHor}.

\subsection{Weak median decomposability}

In the case of bounded distributive lattices $L$, the $n$-ary polynomial functions on $L$ are exactly those which
satisfy the median decomposition formula (\ref{MedDecomposition}); see \cite{Marc}. As we are going to see, in the case of chains, this condition can be relaxed
by restricting the satisfaction of (\ref{MedDecomposition}) by a function $f\colon L^n\to L$ to the vectors of $L_n^{(0,2)}\cup L_n^{(1,3)}$. In the
latter case, we say that $f\colon L^n\to L$ is \emph{weakly median decomposable}.

\begin{lemma}\label{lemma:WMD-RI}
Assume $L$ is a bounded chain and let $f\colon L^{n}\rightarrow L$ be nondecreasing. If $f$ is weakly median decomposable, then it is
$\co{\mathcal{R}}_f$-idempotent.
\end{lemma}

\begin{proof}
Let $c\in\co{\mathcal{R}}_f$ and suppose $f(c,\ldots,c)>c$. Then, by weak median decomposability, we have
$$
f(c,\ldots,c)=\median(f(0,c,\ldots,c), c, f(1,c,\ldots,c))=f(0,c,\ldots,c).
$$
By applying the same argument, we obtain $f(c,\ldots,c)=f(0,0,c,\ldots,c)$, and finally
$$
c<f(c,\ldots,c)=f(\vect{0})\leqslant c,
$$
that is a contradiction. The case $f(c,\ldots,c)<c$ can be dealt with similarly.
\end{proof}

\begin{proposition}\label{prop:WeakMedWeakHom}
Let $L$ be a bounded chain and let $f\colon L^{n}\rightarrow L$ be a nondecreasing function. The following conditions are equivalent:
\begin{enumerate}
\item[(i)] $f$ is weakly median decomposable.

\item[(ii)] $f$ weakly $\co{\mathcal{R}}_f$-min homogeneous and weakly $\co{\mathcal{R}}_f$-max homogeneous.
\end{enumerate}
\end{proposition}

\begin{proof}
We first prove that $(ii)\Rightarrow(i)$. By Theorem~\ref{theorem:WLP-WeakHom}, $f$ is a polynomial function, and thus, by
Theorem~\ref{mainChar}, it is median decomposable. In particular, it is weakly median decomposable.

Now we prove $(i)\Rightarrow(ii)$. We only show that $f$ is weakly $\co{\mathcal{R}}_f$-min homogeneous. The other property can be proved
dually. Let $\vect{x}=\median(c,\vect{e},d)\in L_n^{(0,2)}$, where $c,d\in L$ and $c\leqslant d$, and let $r\in \co{\mathcal{R}}_f$. By
Lemma~\ref{lemma:WMD-RI}, we have $f(r,\ldots,r)=r$.

\begin{itemize}
\item If $r\geqslant d$ then $f(\vect{x}\wedge r)=f(\vect{x})=f(\vect{x})\wedge f(r,\ldots,r)=f(\vect{x})\wedge r$.

\item If $r\leqslant c$ then $f(\vect{x}\wedge r)=f(r,\ldots,r)=f(\vect{x})\wedge f(r,\ldots,r)=f(\vect{x})\wedge r$.

\item Assume $r\in\left]c,d\right[$ and let $K=\{k\in [n]: x_k=d\}$. By weak median decomposability, for any $k\in K$, we have
\begin{eqnarray*}
f(\vect{x}\wedge r) &=& f(\median(c,\vect{e},r))\\
&=& f(\median(c,\vect{e},r)_k^1)\wedge \big(r\vee f(\median(c,\vect{e},r)_k^0)\big)\\
&=& f(\median(c,\vect{e},r)_k^1)\wedge r.
\end{eqnarray*}
By repeating this process, we finally obtain $f(\vect{x}\wedge r)=f(\median(c,\vect{e},1))\wedge r$. Since $f$ is nondecreasing, we have
\begin{eqnarray*}
f(\vect{x}\wedge r) &=& f(\median(c,\vect{e},r))\wedge r \leqslant f(\median(c,\vect{e},d))\wedge r\\
&\leqslant & f(\median(c,\vect{e},1))\wedge r = f(\vect{x}\wedge r)
\end{eqnarray*}
that is $f(\vect{x}\wedge r)=f(\vect{x})\wedge r$.\qedhere
\end{itemize}
\end{proof}

\begin{remark}
Using the binary function $f$ given in Remark~\ref{HomogNotLattice}, we can see that Proposition~\ref{prop:WeakMedWeakHom} does not hold in the
general case of bounded distributive lattices. Indeed, as observed,  $f$ is nondecreasing and both weakly $\co{\mathcal{R}}_f$-min homogeneous
and weakly $\co{\mathcal{R}}_f$-max homogeneous, but $f(b,b)=1\neq b=\median(f(0,b),b,f(1,b))$ which shows that $f$ is not weakly median
decomposable.
\end{remark}

From Proposition~\ref{prop:WeakMedWeakHom} and Theorem~\ref{theorem:WLP-WeakHom}, we obtain the following description of polynomial functions
given in terms of weak median decomposability.

\begin{theorem}\label{theorem:WLP-WeaklyMed}
Let $L$ be a bounded chain. A nondecreasing function $f\colon L^{n}\rightarrow L$ is a polynomial function if and only if it is weakly median
decomposable.
\end{theorem}

\begin{remark}
Note that Theorem~\ref{theorem:WLP-WeaklyMed} does not hold if weak median decomposability would have been defined in terms of vectors in
$L_n^{(0,2)}$ only. To see this, let $L=\{0,c,1\}$ and consider the following nondecreasing function $f\colon L^3\to L$, defined by
$$
f(x_1,x_2,x_3)=
\begin{cases}
1, & \mbox{if $\median(x_1,x_2,x_3)=1$,}\\
c, & \mbox{if $\median(x_1,x_2,x_3)=x_1\wedge x_2\wedge x_3=c$,}\\
0, & \mbox{otherwise}.
\end{cases}
$$
It is easy to this that $f$ is median decomposable for vectors in $L_n^{(0,2)}$, but it is not a polynomial function, e.g., we have $f(0,c,c)=0$
but $f(0,1,1)\wedge c=c$.
\end{remark}

\subsection{Strong idempotency and componentwise range convexity}

Assume $L$ is a bounded chain. By Theorem~\ref{mainChar}, a nondecreasing function $f\colon L^{n}\rightarrow L$ is a polynomial function if and only if it is strongly
idempotent, has a convex range, and a componentwise convex range. Our next result shows that the condition requiring a convex range becomes
redundant in the case when $L$ is a chain, since it becomes a consequence of componentwise range convexity.

\begin{lemma}\label{lemma:ComponentwiseImpliesConv}
Let $L$ be a bounded chain. If a nondecreasing function $f\colon L^{n}\rightarrow L$ has a componentwise convex range, then it has a convex
range.
\end{lemma}

\begin{proof}
Since $L$ is a chain and $f$ has a componentwise convex range, we have
$$
\co{\mathcal{R}}_f =[f(\vect{0}),f(\vect{1})]\subseteq \bigcup_{i=1}^n \, [f(\vect{e}_{\{1,\ldots,i-1\}}),f(\vect{e}_{\{1,\ldots,i\}})]\subseteq
\mathcal{R}_f \subseteq \co{\mathcal{R}}_f.
$$
Therefore $\co{\mathcal{R}}_f=\mathcal{R}_f$ and $f$ has a convex range.
\end{proof}

Using Lemma~\ref{lemma:ComponentwiseImpliesConv}, we obtain the following characterization of polynomial functions which weakens condition
$(iii)$ of Theorem~\ref{mainChar} when $L$ is a chain.

\begin{theorem}\label{thm:ChainStrIdemWLP}
Let $L$ be a bounded chain. A function $f\colon L^{n}\rightarrow L$ is a polynomial function if and only if it is nondecreasing, strongly
idempotent, and has a componentwise convex range.
\end{theorem}

\begin{remark}
None of the conditions provided in Theorem~\ref{thm:ChainStrIdemWLP} can be dropped off. For instance, let $L$ be the real interval $[0,1]$.
Clearly, the unary function $f(x)=x^2$ is nondecreasing and has a componentwise convex range, but it is not strongly idempotent. On the other
hand, the function $f\colon L^2\rightarrow L$ defined by
\[
f(x_1,x_2) =
\begin{cases}
1, & \text{if $x_1=x_2=1$,} \\
0, & \text{otherwise,}
\end{cases}
\]
is nondecreasing and strongly idempotent but it does not have a componentwise convex range, e.g., both $f_{\vect{1}}^1$ and $f_{\vect{1}}^2$ do
not have convex ranges.
\end{remark}

In the special case of real interval lattices, i.e., where $L=[a,b]$ for reals $a\leqslant b$, the property of having a convex range, as well as
the property of having a componentwise convex range, are consequences of continuity. More precisely, for nondecreasing functions
$f\colon [a,b]^n\to\mathbb{R}$, being continuous reduces to being continuous in each variable, and this latter property is equivalent to having a
componentwise convex range. In fact, since polynomial functions are continuous, the condition of having a componentwise convex range can be
replaced in Theorem \ref{thm:ChainStrIdemWLP} by continuity in each variable. Also, by Proposition~\ref{prop:Hom-Id-46} $(iii)$, we can add
continuity and replace $\co{\mathcal{R}}_f$ by $\mathcal{R}_f$ in Theorems~\ref{theorem:WLP-WeakHom} and \ref{theorem:WLP-WeakHomWeakHor}.

\begin{corollary}
Assume that $L$ is a bounded real interval $[a,b]$. A function $f\colon L^{n}\rightarrow L$ is a polynomial function if and only if it is
nondecreasing, strongly idempotent, and continuous (in each variable).
\end{corollary}

\subsection{Comonotonic maxitivity and minitivity}

Let $L$ be a bounded chain. Two vectors $\vect{x},\vect{x}'\in L^n$ are said to be
\emph{comonotonic} if there exists a permutation $\sigma$ on $[n]$ such that $\vect{x},\vect{x}'\in L^n_\sigma$. A function $f\colon L^{n}\rightarrow L$ is said to be
\begin{itemize}
\item \emph{comonotonic minitive} if, for any two comonotonic vectors $\vect{x},\vect{x'}\in L^n$, we have
$$
f(\vect{x}\wedge \vect{x'}) = f(\vect{x})\wedge f(\vect{x'}).
$$

\item \emph{comonotonic maxitive} if, for any two comonotonic vectors $\vect{x},\vect{x'}\in L^n$, we have
$$
f(\vect{x}\vee \vect{x'}) = f(\vect{x})\vee f(\vect{x'}).
$$
\end{itemize}

Note that for any $\vect{x}\in L^n$ and any $c\in L$, we have that $\vect{x}$ and $(c,\ldots ,c)$ are comonotonic and that $\vect{x}\vee c$ and
$[\vect{x}]^c$ are comonotonic. These facts lead to the following result.

\begin{lemma}\label{lemma:ComonotHomog}
Let $L$ be a bounded chain and let $S$ be a nonempty subset of $L$. If a function $f\colon L^{n}\rightarrow L$ is comonotonic minitive (resp.\
comonotonic maxitive), then it is horizontally $S$-minitive (resp.\ horizontally $S$-maxitive). Moreover, if $f$ is $S$-idempotent, then it is
$S$-min homogeneous (resp.\ $S$-max homogeneous).
\end{lemma}

Let $\sigma$ be a permutation on $[n]$. Clearly, every comonotonic minitive (or comonotonic maxitive) function $f\colon L^{n}\rightarrow L$ is
nondecreasing on the standard simplex $L^n_\sigma$. The following lemma shows that this fact can be extended to the whole domain $L^n$.

\begin{lemma}\label{lemma:ComonotNonDec}
Let $L$ be a bounded chain. If $f\colon L^{n}\rightarrow L$ is comonotonic minitive or comonotonic maxitive, then it is nondecreasing.
Furthermore, every nondecreasing unary function is comonotonic minitive and comonotonic maxitive.
\end{lemma}

\begin{proof}
To see that the last claim holds just note that, on any chain $L$, we necessarily have $x\leqslant y$ or $x\geqslant y$ for every $x,y\in L$.
For instance, if $x\leqslant y$ then we have $f(x\wedge y)=f(x)=f(x)\wedge f(y)$ and $f(x\vee y)=f(y)=f(x)\vee f(y)$.

We now prove the first claim for comonotonic minitive functions. The case of comonotonic maxitive functions is shown similarly. Let  $f\colon
L^{n}\rightarrow L$ be a comonotonic minitive function and consider $\vect{x},\vect{x'}\in L^n$ such that $\vect{x}\leqslant \vect{x'}$ and
$\vect{x}\neq \vect{x'}$. We show that $f(\vect{x})\leqslant f(\vect{x'})$. For each $i\in [n]$, we denote by $\vect{y}^i$ the vector in $L^n$
whose $j$th component is $x'_j$ if $j\leqslant i$, and $x_j$ otherwise. As a matter of convenience, let $\vect{y}^0=\vect{x}$. Clearly, we have
$$
\vect{x}=\vect{y}^0\leqslant \vect{y}^1\leqslant \cdots \leqslant \vect{y}^{n-1}\leqslant \vect{y}^n=\vect{x'}.
$$
Let $k\in [n]$. If $\vect{y}^{k-1}$ and $\vect{y}^k$ are comonotonic, then $f(\vect{y}^{k-1})=f(\vect{y}^{k-1}\wedge\vect{y}^k)=f(\vect{y}^{k-1})\wedge f(\vect{y}^k)\leqslant f(\vect{y}^k)$. Otherwise, either there is $j>k$ such that $x_k<x_j<x'_k$ or there is $j<k$ such that $x_k<x'_j<x'_k$. Let $\vect{y'}$ be the vector obtained
from  $\vect{y}^k$ by replacing the $k$th component with $x_j$ if $j>k$ and with $x'_j$ if $j<k$. Then both $\{\vect{y}^{k-1},\vect{y'}\}$ and
$\{\vect{y'},\vect{y}^k\}$ constitute pairs of comonotonic vectors and, since $\vect{y}^{k-1}<\vect{y'}<\vect{y}^k $, it follows that
$f(\vect{y}^{k-1})\leqslant f(\vect{y'})\leqslant f(\vect{y}^k)$. Since the same argument holds for any $k\in [n]$, we have that
$f(\vect{x})\leqslant f(\vect{x'})$.
\end{proof}

We now have the following characterization of polynomial functions.

\begin{theorem}\label{theorem:WLP-comonot}
Let $L$ be a bounded chain and let $f\colon L^{n}\rightarrow L$ be a function. The following conditions are equivalent:
\begin{enumerate}
\item[(i)] $f$ is a polynomial function.

\item[(ii)] $f$ is weakly $\co{\mathcal{R}}_f$-min homogeneous and comonotonic maxitive.

\item[(iii)] $f$ is comonotonic minitive and weakly $\co{\mathcal{R}}_f$-max homogeneous.

\item[(iv)] $f$ is $\co{\mathcal{R}}_f$-idempotent, weakly horizontally $\co{\mathcal{R}}_f$-minitive, and comonotonic maxitive.

\item[(v)] $f$ is $\co{\mathcal{R}}_f$-idempotent, comonotonic minitive, and weakly horizontally $\co{\mathcal{R}}_f$-maxitive.

\item[(vi)] $f$ is $\co{\mathcal{R}}_f$-idempotent, comonotonic minitive, and comonotonic maxitive.
\end{enumerate}
\end{theorem}

\begin{proof}
Using distributivity and the first equality in (\ref{DNF1234}) of Proposition~\ref{SimplexDNF}, it can be easily verified that every polynomial
function is comonotonic maxitive. By the dual argument, it follows that every polynomial function is also comonotonic minitive. Thus we have
$(i)\Rightarrow (vi)$. The implications $(vi)\Rightarrow (v)$ and $(vi)\Rightarrow (iv)$ immediately follow from Lemma~\ref{lemma:ComonotHomog}.
Then, the implications $(iv)\Rightarrow (ii)$ and $(v)\Rightarrow (iii)$ immediately follow from Lemmas~\ref{lemma:Weak15682} and
\ref{lemma:ComonotNonDec}. Finally, the implications $(ii)\Rightarrow (i)$ and $(iii)\Rightarrow (i)$ follow from
Lemmas~\ref{lemma:WeaklyMinMaxRangeIdem}, \ref{lemma:ComonotHomog}, \ref{lemma:ComonotNonDec}, and Theorem~\ref{theorem:WLP-WeakHom}.
\end{proof}

\begin{remark}
\begin{enumerate}
\item[(i)] As already observed in the remark following Theorem~\ref{theorem:WLP-WeakHomWeakHor}, the weak horizontal
$\co{\mathcal{R}}_f$-minitivity (resp.\ weak horizontal $\co{\mathcal{R}}_f$-maxitivity) can be replaced with weak horizontal $L$-minitivity
(resp.\ weak horizontal $L$-maxitivity) in the assertions $(iv)$--$(v)$ of Theorem~\ref{theorem:WLP-comonot}.

\item[(ii)] The condition requiring  $\co{\mathcal{R}}_f$-idempotency is necessary in conditions $(iv)$--$(vi)$ of
Theorem~\ref{theorem:WLP-comonot}. For instance, let $L$ be the unit interval $[0,1]$. Clearly, the unary function $f(x)=x^2$ is nondecreasing
and thus comonotonic minitive and comonotonic maxitive. By Lemma~\ref{lemma:ComonotHomog}, it is also horizontally $\co{\mathcal{R}}_f$-minitive
and horizontally $\co{\mathcal{R}}_f$-maxitive. However, it is not a polynomial function.

\item[(iii)] The concept of comonotonic vectors appeared as early as 1952 in Hardy et al.~\cite{HarLitPol52}. Comonotonic minitivity and
maxitivity were introduced in the context of Sugeno integrals by de Campos et al.~\cite{deCLamMor91}. An interpretation of these properties was
given by Ralescu and Ralescu~\cite{RalRal97} in the framework of aggregation of fuzzy subsets.
\end{enumerate}
\end{remark}

\section{Some special classes of polynomial functions}

In this final section, we consider two noteworthy subclasses of polynomial functions, namely, those of discrete Sugeno integrals and of term
functions, and provide their characterizations. Further subclasses, such as those of symmetric polynomial functions and weighted minimum and
maximum functions, were investigated and characterized in \cite{CouMar1}.

\subsection{Discrete Sugeno integrals}

A function $f\colon L^{n}\rightarrow L$ is said to be \emph{idempotent} if it is $L$-idempotent.

\begin{fact}\label{sugeno}
A polynomial function is $\{0,1\}$-idempotent if and only if it is idempotent.
\end{fact}

In \cite{Marc}, $\{0,1\}$-idempotent polynomial functions are referred to as \emph{discrete Sugeno integrals}. They coincide
exactly with those functions $\mathcal{S}_{\mu}\colon L^{n}\rightarrow L$ for which there is a fuzzy measure $\mu$ such that
$$
\mathcal{S}_{\mu}(\vect{x})=\bigvee_{I\subseteq [n]}\big(\mu(I)\wedge \bigwedge_{i\in I} x_i\big).
$$
(A \emph{fuzzy measure} $\mu$ is simply a set function $\mu \colon 2^{[n]}\rightarrow L$ satisfying $\mu(I)\leqslant \mu(I')$ whenever
$I\subseteq I'$, and $\mu(\varnothing)=0$ and $\mu([n])=1$.)

For idempotent polynomial functions, Proposition~\ref{SimplexDNF} reduces to the following statement.

\begin{proposition}
Let $L$ be a bounded chain and let $f\colon L^{n}\rightarrow L$ be a function. The following conditions are equivalent:
\begin{enumerate}
\item[(i)] $f$ is a discrete Sugeno integral.

\item[(ii)] For any permutation $\sigma$ on $[n]$ and every $\vect{x}\in L^n_\sigma$, we have
\begin{eqnarray*}
 f(\vect{x}) &=& \bigvee_{i\in [n]}\big (\alpha_f(S^{\uparrow}_\sigma (i))\wedge x_{\sigma(i)}\big )
= \bigwedge_{i\in [n]}\big (\alpha_f(S^{\uparrow}_\sigma (i+1))\vee x_{\sigma(i)}\big )\\
&=& \median\big(x_1,\ldots ,x_n, \alpha_f(S^{\uparrow}_\sigma (2)),\ldots ,\alpha_f(S^{\uparrow}_\sigma (n))\big).
\end{eqnarray*}

\item[(iii)] For any permutation $\sigma$ on $[n]$ and every $\vect{x}\in L^n_\sigma$, we have
\begin{eqnarray*}
f(\vect{x}) &=& \bigvee_{i\in [n]}\big (\beta_f(S^{\downarrow}_\sigma (i-1))\wedge x_{\sigma(i)}\big ) =
\bigwedge_{i\in [n]}\big (\beta_f(S^{\downarrow}_\sigma (i))\wedge x_{\sigma(i)}\big )\\
&=& \median\big(x_1,\ldots ,x_n, \beta_f(S^{\downarrow}_\sigma (1)),\ldots ,\beta_f(S^{\downarrow}_\sigma (n-1))\big).
\end{eqnarray*} 
\end{enumerate}
\end{proposition}

The following proposition shows how polynomial functions relate to Sugeno integrals; see \cite[Proposition~12]{Marc}.

\begin{proposition}\label{prop:sug}
For any polynomial function $f\colon L^{n}\rightarrow L$ there is a fuzzy measure $\mu \colon 2^{[n]}\rightarrow L$ such that
$f=\langle\mathcal{S}_{\mu}\rangle_f$.
\end{proposition}

We say that a function $f\colon L^{n}\rightarrow L$ is \emph{Boolean min homogeneous} (resp.\ \emph{Boolean max homogeneous}) if (\ref{minHom})
(resp.\ (\ref{maxHom})) holds for every $\vect{x}\in \{0,1\}^n$ and every $c\in L$. Note that every weakly $L$-min homogeneous (resp.\ weakly
$L$-max homogeneous) function is Boolean min homogeneous (resp.\ Boolean max homogeneous).

\begin{lemma}\label{lemma:4985}
If a function $f\colon L^n\to L$ is Boolean min homogeneous and Boolean max homogeneous, then it is idempotent.
\end{lemma}

\begin{proof}
For any $c\in L$, we have $c\leqslant f(\vect{0})\vee c=f(\vect{0}\vee c)=f(\vect{1}\wedge c)=f(\vect{1})\wedge c\leqslant c$.
\end{proof}

The following result provides a variant of Theorem~\ref{theorem:WLP-WeakHom}.

\begin{theorem}\label{theorem:Sug-WeakHom}
Let $L$ be a bounded chain. A function $f\colon L^{n}\rightarrow L$ is a discrete Sugeno integral if and only if it is nondecreasing, Boolean
min homogeneous, and Boolean max homogeneous.
\end{theorem}

\begin{proof}
Since any discrete Sugeno integral is comonotonic minitive and idempotent, we have $f(\vect{e}\wedge c)=f(\vect{e})\wedge
f(c,\ldots,c)=f(\vect{e})\wedge c$ for every $\vect{e}\in \{0,1\}^n$ and every $c\in L$. Thus, any discrete Sugeno integral is Boolean min
homogeneous. We can prove similarly that it is also Boolean max homogeneous.

To show that the conditions are sufficient, let $\vect{x}\in L^n$. By nondecreasing monotonicity and binary min homogeneity, for every
$I\subseteq [n]$ we have
$$
f(\vect{x}) \geqslant f\big(\vect{e}_I \wedge (\bigwedge_{i\in I}x_i)\big) = f(\vect{e}_I) \wedge \big(\bigwedge_{i\in I}x_i\big)
$$
and thus $f(\vect{x})\geqslant \bigvee_{I\in [n]} (f(\vect{e}_I) \wedge (\bigwedge_{i\in I}x_i))$. The converse inequality can be obtained by
following exactly the same steps as in the proof of Theorem~\ref{theorem:WLP-WeakHom}.
%
%
\end{proof}

%
%
%
%
%
%
%
%
%
%
%
%

\begin{remark}
\begin{enumerate}
\item[(i)] Theorem~\ref{theorem:Sug-WeakHom} as well as the characterization of the discrete Sugeno integrals obtained by combining
$\{0,1\}$-idempotency with $(vi)$ in Theorem~\ref{theorem:WLP-comonot} were presented in the case of real variables in \cite[{\S}4.3]{Mar98}; see also \cite{Mar00c}.

\item[(ii)] Even though Theorem~\ref{theorem:Sug-WeakHom} can be derived from condition $(ii)$ of Theorem~\ref{theorem:WLP-WeakHomWeakHor} by
simply modifying the two homogeneity properties, to proceed similarly with conditions $(iii)$ and $(iv)$, it is necessary to add the conditions
of $\{1\}$-idempotency and $\{0\}$-idempotency, respectively (and similarly with conditions $(ii)$ and $(iii)$ of
Theorem~\ref{theorem:WLP-comonot}). To see this, let $L$ be a bounded chain with at least three elements and consider the unary functions
$f(x)=x\wedge d$ and $g(x)=x\vee d$, where $d\in L\setminus\{0,1\}$. Clearly, $f$ is $L$-min homogeneous and horizontally $L$-maxitive and $g$
is $L$-max homogeneous and horizontally $L$-minitive. However, neither $f$ nor $g$ is a discrete Sugeno integral. To see that these additions
are sufficient, just note that $L$-min homogeneity (resp.\ $L$-max homogeneity) implies $\{0\}$-idempotency (resp.\ $\{1\}$-idempotency).

\item[(iii)] It was shown in \cite[{\S}2.2.3]{Mar98} that, when $L$ is a chain, a nondecreasing and idempotent function $f\colon
L^{n}\rightarrow L$ is Boolean min homogeneous (resp.\ Boolean max homogeneous) if and only if we have $f(\vect{e}\wedge c)\in\{f(\vect{e}),c\}$
(resp.\ $f(\vect{e}\vee c)\in\{f(\vect{e}),c\}$) for every $\vect{e}\in \{0,1\}^n$ and every $c\in L$.
\end{enumerate}
\end{remark}

\subsection{Lattice term functions}

A function $f\colon L^{n}\rightarrow L$ is said to be
\begin{itemize}
\item \emph{conservative} if, for every $\vect{x}\in L^n$, we have
\begin{equation}\label{eq:discretiz}
f(\vect{x})\in \{x_1,\ldots,x_n\}.
\end{equation}

\item \emph{weakly conservative} if (\ref{eq:discretiz}) holds for every $\vect{x}\in \{0,1\}^n$.
\end{itemize}

\begin{remark}
Conservative functions (also called \emph{quasi-projections}) were defined in the binary case by Pouzet et al.~\cite{PouRosSt96}.
\end{remark}

\begin{proposition}
Let $f\colon L^{n}\rightarrow L$ be a function. The following conditions are equivalent:
\begin{enumerate}
\item[(i)] $f$ is conservative.

\item[(ii)] For every nonempty $S\subseteq L$, we have $f(S^n)\subseteq S$.

\item[(iii)] For every nonempty $S\subseteq L$ and every $\vect{x}\in L^n$, if $f(\vect{x})\in S$ then there exists $i\in [n]$ such that $x_i\in
S$.
\end{enumerate}
\end{proposition}

\begin{proof}
Assume that $f$ is conservative. For every nonempty $S\subseteq L$ and every vector $\vect{x}\in S^n$, we have $f(\vect{x})\in S$, which proves
that $(i)\Rightarrow (ii)$. Now, assume that $(ii)$ holds and suppose that there exist $S\varsubsetneq L$ and $\vect{x}\in (L\setminus S)^n$
such that $f(\vect{x})\in S$. By $(ii)$, we have $f(\vect{x})\in L\setminus S$, a contradiction. Thus $(ii)\Rightarrow (iii)$. Finally, assume
that $(iii)$ holds and that $f$ is not conservative, that is, there exists $\vect{x}\in L^n$ such that $f(\vect{x})\notin \{x_1,\ldots,x_n\}$.
Then, choosing $S=L\setminus\{x_1,\ldots,x_n\}$ contradicts $(iii)$. This proves $(iii)\Rightarrow (i)$.
\end{proof}

Clearly, every conservative function is idempotent. Similarly, every weakly conservative function is $\{0,1\}$-idempotent. As observed in
\cite{CouMar1}, the term functions are exactly the weakly conservative discrete Sugeno integrals. Similarly, we can readily see that, when $L$ is
a chain, the term functions $f\colon L^{n}\rightarrow L$ are exactly the conservative discrete Sugeno integrals.

\begin{theorem}\label{mainChar3}
Let $L$ be a bounded chain and let $f\colon L^{n}\rightarrow L$ be a discrete Sugeno integral. The following conditions are equivalent:
\begin{enumerate}
\item[(i)] $f$ is a term function.

\item[(ii)] $f$ is conservative.

\item[(iii)] $f$ is weakly conservative.
\end{enumerate}
\end{theorem}

\begin{remark}
Not all nondecreasing and conservative functions are term functions. For instance, if $L$ is the real unit interval $[0,1]$, the binary function
$f\colon L^2\to L$, given by $f(x_1,x_2)=x_1\vee x_2$ on $[0.5,1]^2$ and by $f(x_1,x_2)=x_1\wedge x_2$ everywhere else, is nondecreasing and
conservative, but it is not a term function.
\end{remark}


\end{document}